\documentclass{article}

\usepackage{hyperref}
\usepackage{amsmath}
\usepackage{amsthm}
\usepackage{amssymb}
\usepackage{cite}
\usepackage{amsmath,amsthm}
\usepackage{amsfonts}
\usepackage{amssymb}
\usepackage{mathtools}
\usepackage{tikz}
\usepackage{lipsum}
\newcommand\blfootnote[1]{%
  \begingroup
  \renewcommand\thefootnote{}\footnote{#1}%
  \addtocounter{footnote}{-1}%
  \endgroup
}
\usepackage{enumitem}

\usepackage{dirtytalk} 
\usepackage{import}
\newtheorem{theorem}{Theorem}[section]

\newtheorem{example}[theorem]{Example}
\newtheorem{conjecture}[theorem]{Conjecture}
\newtheorem{corollary}[theorem]{Corollary}
\newtheorem{proposition}[theorem]{Proposition}
\newtheorem{lemma}[theorem]{Lemma}
\newtheorem{problem}[theorem]{Problem}

\usepackage{lineno}

\begin{document}

\title{\bf Siblings of Direct Sums of Chains}

\author{Davoud Abdi\blfootnote{This paper is a project as part of author's thesis under supervision of Dr. Claude Laflamme and Dr. Robert Woodrow at the Department of Mathematics and Statistics, University of Calgary, Calgary, AB, Canada (2017-2022).}} 

\maketitle              

\begin{abstract}
We prove that a countable direct sum of chains has  one, countably many or else continuum many isomorphism classes of siblings. This proves Thomass\'e's conjecture for such structures. Further, we show that a direct sum of chains of any cardinality has one or infinitely many siblings, up to isomorphism. 
\end{abstract}


\section{Introduction}

Two structures $R$ and $S$ are called {\em equimorphic} or {\em siblings}, denoted by $R\approx S$, if they mutually embed in each other. The connection between equimorphism and isomorphism goes back to a fundamental theorem in set theory namely the Cantor-Schr\"{o}der-Bernstein Theorem asserting that if there are mutual injections between two sets $A$ and $B$, then there is a bijection between $A$ and $B$. Similarly, in the category of vector spaces, if there are mutual injective linear transformations between two vector spaces over the same field, then the vector spaces are isomorphic. Nonetheless, the rational numbers considered as a chain has continuum many siblings, up to isomorphism. Therefore, it is a natural problem to investigate relations, understand all their siblings and as a first approach count the siblings (up to isomorphism). 

The number of isomorphism classes of siblings of a structure $R$ is called the {\em sibling number} of $R$ and we denote it by $Sib(R)$. A relation $R$ is a pair $(V,E)$ where $V$ is a non-empty set, called the {\em domain} of $R$, and $E\subseteq V^n$ for some positive integer $n$. Then, the integer $n$ is called the {\em arity} of $E$. A {\em binary relation} is a relation $(V,E)$ where the arity of $E$ is 2. A relation $(V,E)$ has cardinality $\kappa$ if $|V|=\kappa$. Thomass\'e \cite{TH1} made the following conjecture:

\begin{conjecture} [Thomass\'e's Conjecture, \cite{TH1}]
Let R be a countable relation. Then, $Sib(R)=1$ or $\aleph_0$ or $2^{\aleph_0}$. 
\end{conjecture}

There is also an alternate version of the conjecture above as follows:

\begin{conjecture} [The Alternate Thomass\'e Conjecture]
Let R be a relation of any cardinality. Then, $Sib(R)=1$ or $\infty$. 
\end{conjecture}

The alternate Thomass\'{e} conjecture is connected to a conjecture of Bonato-Tardif \cite{BT} ({\em The Tree Alternative Conjecture}) stating that the sibling number of a tree of any cardinality is one or infinite in the category of trees. The connection between these two conjectures is through this fact that each sibling of a tree $T$ is a tree if and only if $T\oplus 1$, the graph obtained by adding an isolated vertex to $T$, does not embed in $T$. Therefore, for a tree $T$ not equimorphic to $T\oplus 1$, the tree alternative conjecture and the alternate Thomass\'e conjecture are equivalent. The tree alternative conjecture has been verified for rayless trees by Bonato and Tardif \cite{BT}, for rooted trees by Tyomkyn \cite{TY} and for scattered trees by Laflamme, Pouzet and Sauer \cite{LPS}. In a parallel direction, the alternate Thomass\'e conjecture was proved for rayless graphs by Bonato, Bruhn, Diestel and Spr\"ussel \cite{BBDS}, for countable $\aleph_0$-categorical structures by Laflamme, Pouzet, Sauer and Woodrow \cite{LPSW}, for countable cographs by Hahn, Pouzet and Woodrow \cite{HPW} and for countable universal theories by Braunfeld and Laskowski \cite{BL}. 

Tateno \cite{TAT} claimed a counterexample to the tree alternative conjecture in his thesis. 
Abdi, Laflamme, Tateno and Woodrow \cite{ALTW} revisited and verified Tateno's claim and provided locally finite trees having an arbitrary finite number of siblings. Moreover, using an adaptation, they \cite{ALTW} provided countable partial orders with a similar conclusion which disproves Thomass\'e's conjecture. It turns out that both conjectures of Bonato-Tardif and Thomass\'e  are false. 
This is a major development in the programme of understanding siblings of a given mathematical structure. We now know that the simple dichotomy one vs infinite for the number of siblings is no longer valid, and the situation is much more intricate. While counting the number of siblings provides a good first insight into the siblings of a structure, in particular understanding which structures exactly do satisfy the conjectures, the equimorphy programme is now ready to move on and focus on the actual structure of those siblings. 
Laflamme, Pouzet and Woodrow verified the following (see \cite{LPW} Theorem 3.1 and Corollary 3.6). 

\begin{theorem} [\cite{LPW}] \label{Chaindich}
Let C be a chain. Then $Sib(C)=1$ or $\infty$. Moreover, if C is countable, then $Sib(C)=1$, $\aleph_0$ or $2^{\aleph_0}$.
\end{theorem}

As a natural case towards the equimorphy programme for partial orders, we consider partially ordered sets whose components are chains and verify both Thomass\'{e}'s conjecture and its alternate version for such structures.

\section{Definitions and Examples}

A \textit{poset (partially ordered set)} is a binary relation $P=(V,\leq)$ where $\leq$, the {\em order} of $P$, is  reflexive, antisymmetric and transitive. If a poset $P$ is given, by $\leq_P$, we mean the order of $P$. Let $P=(V,\leq)$ be a poset. By an element $x$ of $P$ we mean $x\in V$. Two elements $x, y \in P$ are called \textit{comparable} if $x \leq y$ or $y \leq x$, otherwise they are \textit{incomparable}, denoted by $x \perp y$. 
A \textit{chain} ({\em linearly ordered set}), resp an \textit{antichain}, is a poset whose elements are pairwise comparable, resp incomparable. Any substructure of a chain is called a {\em subchain}. For $n<\omega$, we denote a chain of size $n$ by $C^n$.

Let $\{C_i\}_{i\in I}$ be a family of chains. Their {\em direct sum}, DSC in short, denoted by $\mathcal{D}=\bigoplus_{i\in I} C_i$, is the disjoint union of the $C_i$ such that for two elements $x, y \in \mathcal{D}$, $x\leq_\mathcal{D} y$ if and only if for some $i\in I$, $x, y\in P_i$ and $x\leq_{P_i} y$. Let $\mathcal{D}=\bigoplus_{i \in I}C_i$ be a DSC. Each chain $C_i$ is called a {\em component} of $\mathcal{D}$.
A component of $\mathcal{D}$ is \textit{trivial} if it consists of a single element. 
By a countable direct sum of chains $\mathcal{D}$ we mean that the number of components of  $\mathcal{D}$ is countable and that every component of $\mathcal{D}$ is countable.

\begin{example}\label{2.1}
$\mathcal{D}=\bigoplus_{\omega_1} \omega \oplus \bigoplus_\omega (\omega + 1) \oplus T$ where $T$ is the direct sum of the singletons of $\mathcal{D}$ with $|T|=\aleph_1$. 
\end{example}

\begin{example}\label{2.2}
$\mathcal{D}=\bigoplus_{n < \omega} \bigoplus_{\omega_n} C^n \oplus T$ where $T$ is the direct sum of the singletons of $\mathcal{D}$ with $|T|=\aleph_\omega$. 
\end{example}

\begin{example}\label{2.3}
$\mathcal{D}=\bigoplus_{\omega_1} C^1 \oplus \bigoplus_\omega C^3$. 
\end{example}

\begin{example}\label{2.4}
$\mathcal{D}=\bigoplus_{\omega_3} C^1 \oplus \bigoplus_{\omega_2} C^2 \oplus \bigoplus_{\omega_1} C^3 \oplus \bigoplus_{\omega_1} \omega$. 
\end{example}

\section{Structural Results} \label{Structural}

As the first step in counting the siblings of a DSC, we show that two equimorphic DSCs have the same number of components. 

\begin{lemma} \label{InjectiononI}
Let $\mathcal{D}=\bigoplus_{i\in I} C_i$ and $\mathcal{D}'=\bigoplus_{j\in J}C_j'$ be two direct sums of chains. Each embedding from $\mathcal{D}$ to $\mathcal{D}'$ induces an injection from I to J. In particular if $\mathcal{D}\approx \mathcal{D}'$, then $|I|=|J|$. 
\end{lemma}

\begin{proof}
Let $f:\mathcal{D}\hookrightarrow\mathcal{D}'$ be an embedding. Let $i\in I$ be given. First note that since $C_i$ is connected, $f(C_i)\subseteq C_j$ for some $j\in J$. For $i\in I$ define $\hat{f}(i)=j$ where $f(C_i)\subseteq C_j$. Since the components of $\mathcal{D}'$ are disjoint, $\hat{f}$ is well-defined. Now suppose that for some $i\neq i'$, $\hat{f}(i)=\hat{f}(i')=j$. Then, $f:C_i\oplus C_{i'}\hookrightarrow C_j$ which is not possible because $C_i\oplus C_{i'}$ contains two incomparable elements. Thus, $\hat{f}$ is an injection from $I$ to $J$.  

Now let $f:\mathcal{D}\to \mathcal{D}'$ and $g:\mathcal{D}'\to\mathcal{D}$ be embeddings and  $\hat{f}:I\to J$ and $\hat{g}:J\to I$ be defined as above. We know that both $\hat{f}$ and $\hat{g}$ are injective. Since $I$ and $J$ are independent sets, we then have $|I|=|J|$ by the
Cantor-Schr\"{o}der-Bernstein theorem.
\end{proof}

The proof of Lemma \ref{InjectiononI} implies that a sibling of a direct sum of chains must be a direct sum of chains, and that
two direct sum of chains are isomorphic if and only if there is a bijection
between the index sets and an isomorphisms between corresponding
chains.

Let $\mathcal{D}$ be a DSC. 
We may represent $\mathcal{D}$ as $\bigoplus_{i\in I}C_i\oplus T$  where the $C_i$ are non-trivial components and $T$ is the direct sum of the singletons which is an antichain. Let $\mathcal{D}'$ be a sibling of $\mathcal{D}$. We notice that $\mathcal{D}'$ is isomorphic to an induced substructure of $\mathcal{D}$. Hence, $\mathcal{D}'=\bigoplus_{j\in J} C'_j\oplus T'$ where the $C_j'$ are non-trivial and $T'$ is the direct sum of the singletons. Clearly, a non-trivial component does not embed in a singleton. Thus, each $C_i$ embeds in some $C_j'$. By symmetry, each $C_j'$ embeds in some $C_i$. Therefore, $\bigoplus_{i\in I} C_i\approx \bigoplus_{j\in J} C_j'$ and by Lemma \ref{InjectiononI}, we have $|I|=|J|$, however,  as we will see in Lemma \ref{Increasingsequence}, it is possible to have $T$ not isomorphic to $T'$. As a case in point, we have $\bigoplus_{\omega}C^2 \approx\bigoplus_{\omega}C^2\oplus 1$.

\begin{lemma} \label{FinitenontrivialD}
Let $\mathcal{D}$ be a DSC containing only finitely many non-trivial components. If every component of $\mathcal{D}$ has just one sibling, then so does $\mathcal{D}$.  
\end{lemma}

\begin{proof}
Let $\mathcal{D}=\bigoplus_{i < n}C_i\oplus T$ satisfy the conditions of the lemma where the $C_i$ are non-trivial and $T$ is the direct sum of the singletons. Let $f$ be an embedding of $\mathcal{D}$ and $\hat{f}$ be defined as in Lemma \ref{InjectiononI} which is injective. Let $i < n$ be given. Since $n$ is finite, by iteration and the fact that a non-trivial component does not embed in a singleton, we get an integer $k_i$ such that $\hat{f}^{k_i}(i)=i$. Now let $x\in T$ be given. If $f(x)\in C_i$ for some $i<n$, then picking some element $y\in C_{\hat{f}^{k_i-1}}$, we have $f(x), f(y)\in C_i$, a contradiction because $x\perp y$. Thus, $f(T)\subseteq T$.  This also implies that $f\left(\bigoplus_{i<n} C_i\right)\subseteq \bigoplus_{i<n} C_i$. 

Let $\mathcal{D}'\subseteq \mathcal{D}$ be a sibling of $\mathcal{D}$ and $f$ an embedding of $\mathcal{D}$ such that $f(\mathcal{D})\hookrightarrow \mathcal{D}'\hookrightarrow \mathcal{D}$. By the argument above, $\mathcal{D}'$ is of the form $\mathcal{D}'=\bigoplus_{i<n} C_i'\oplus T'$ where $f^{k_i}(C_i)\subseteq C_i' \subseteq C_i$ for each $i$ and $f(T)\subseteq T'\subseteq T$. Therefore, $T'\approx T$ and for each $i$, $C_i'\approx C_i$. Since $T$ and $T'$ are antichains, we have $T'\cong T$ and by assumption, for each $i<n$, $C_i'\cong C_i$.  Hence, $\mathcal{D}\cong \mathcal{D}'$ and  $Sib(\mathcal{D})=1$.  
\end{proof}

\begin{lemma} \label{Infinitesiblingcomponent}
Let $\mathcal{D}=\bigoplus_{i\in I}C_i$ be a DSC. Then $Sib(\mathcal{D})\geq \max \{Sib(C_i) : i\in I\}$. In particular, if for some $i$ we have $Sib(C_i)=\infty$, then $Sib(\mathcal{D})=\infty$.  
\end{lemma} 

\begin{proof}
Given $k\in I$, let $Sib(C_k)=\lambda$ and $\{C_{kj}\}_{j< \lambda}$ be a family of pairwise non-isomorphic siblings of $C_k$. For every $j< \lambda$, define $\mathcal{D}_j=\bigoplus_{i\in I}D_i$ where 
$$D_i=\begin{cases} 
C_i & \text{if}\   C_i\not\approx C_k \\
C_{kj} & \text{if} \ C_i\approx C_k.
\end{cases}$$
Since for each $i$, $C_i\approx D_i$, we have $\mathcal{D}\approx \mathcal{D}_j$. Now assume that for two distinct $j, j'<\lambda$ we have $\mathcal{D}_j\cong \mathcal{D}_{j'}$ by an isomorphism $f$. By Lemma \ref{InjectiononI}, $f$ induces a bijection $\hat{f}$ on $I$. We have $C_{kj}=D_k\cong D_{\hat{f}(k)}$. This also implies that $D_{\hat{f}(k)}=C_{kj'}$. But then we have $C_{kj}\cong C_{kj'}$, a contradiction. Therefore, 
The  $\mathcal{D}_j$ are pairwise non-isomorphic siblings of $\mathcal{D}$.  Hence, $Sib(\mathcal{D})\geq \lambda$. 
\end{proof}

\begin{lemma}  \label{Increasingsequence}
Suppose that a direct sum of chains $\mathcal{D}$ has countably many trivial components and that there is an infinite sequence $(C_n)_{n<\omega}$ of non-trivial components of $\mathcal{D}$ which is increasing w.r.t embeddability. Then $Sib(\mathcal{D})=\infty$.  
\end{lemma}

\begin{proof}
Let $\mathcal{N}$ be the direct sum of the non-trivial components of $\mathcal{D}$ distinct from the $C_n$ and $T$ the direct sum of the singletons. We have $\mathcal{D}=\bigoplus_{n<\omega}C_n\oplus \mathcal{N} \oplus T$. By assumption, $|T|\leq \aleph_0$.  Therefore, $T=A^\lambda$ for some $\lambda\leq \aleph_0$. Since $\bigoplus_{n<\omega} C_n$ embeds in $\bigoplus_{n<\omega} C_{2n}$ and $A^\lambda$ embeds in $\bigoplus_{n<\omega} C_{2n+1}$, it follows that $\mathcal{D}$ embeds in $\bigoplus_{n<\omega} C_n\oplus \mathcal{N}=\mathcal{D}'$ which is the direct sum of the non-trivial components of $\mathcal{D}$. Now by a similar argument, for every $m\geq 1$, $\mathcal{D}_m:=\mathcal{D}'\oplus A^m$ is a sibling of $\mathcal{D}'$ and consequently of $\mathcal{D}$. Since each $\mathcal{D}_m$ has exactly $m$ singletons, they are pairwise non-isomorphic. Hence,  $Sib(\mathcal{D})=\infty$.  
\end{proof}

A {\em quasi-order} is a binary relation $Q=(V,\leq)$ where $\leq$ is reflexive and transitive. A quasi-order $Q=(V,\leq)$ is called a \textit{well-quasi-order} (wqo) if for any infinite sequence $(q_n)_{n<\omega}$ of elements of $Q$, there exist $m<n$ such that $q_m\leq q_n$. By Ramsey's Theorem, any infinite sequence of elements of a wqo contains an infinite increasing subsequence (\cite{TH2}). 
It is easy to verify that the class of chains is quasi-ordered under embeddability. Laver proved that the embeddability relation is more well behaved in the case of countable chains. The following theorem is an answer to a conjecture of Fra\"{i}ss\'{e}.

\begin{theorem} [Laver \cite{LAV1}] \label{Laver}
The class of countable chains is wqo under embeddability. 
\end{theorem}

By Theorem \ref{Laver} we can deduce the following.

\begin{proposition} \label{Countabletrivial}
Suppose that a direct sum of chains $\mathcal{D}$ contains infinitely many non-trivial countable components. If the number of the singletons of $\mathcal{D}$ is countable, then $Sib(\mathcal{D})=\infty$. 
\end{proposition}  

\begin{proof}
Let $\mathcal{D}$ be a DSC with infinitely many non-trivial countable components. Then by Theorem \ref{Laver}, there is an infinite sequence of non-trivial countable components of $\mathcal{D}$ which is increasing w.r.t embeddability. If the number of singletons of $\mathcal{D}$ is countable, then the statement follows by Lemma \ref{Increasingsequence}.  
\end{proof}

\begin{corollary} \label{Pcountableinfinite}
If $\mathcal{D}$ is a countable DSC with infinitely many non-trivial components, then $Sib(\mathcal{D})=\infty$.  
\end{corollary} 

For a countable DSC, we can deduce the alternate Thomass\'e conjecture. 

\begin{theorem} \label{DichotomyPcountable} 
If $\mathcal{D}$ is a countable DSC, then $Sib(\mathcal{D})=1$ or $\infty$.
\end{theorem}

\begin{proof}
Let $\mathcal{D}$ be a countable DSC. If $\mathcal{D}$ contains only finitely many non-trivial components, then $\mathcal{D}$ has one or infinitely many siblings by Lemmas \ref{FinitenontrivialD} and \ref{Infinitesiblingcomponent}. If $\mathcal{D}$ has infinitely many non-trivial components, then considering Corollary \ref{Pcountableinfinite} the result follows. 
\end{proof}

\section{Siblings of Countable DSCs} \label{CountableDSC}

Theorem \ref{DichotomyPcountable} asserts that the alternate Thomass\'{e} conjecture holds for a countable DSC. 
In order to prove Thomass\'{e}'s Conjecture for a countable direct sum of chains $\mathcal{D}$, we need to determine the sibling number of $\mathcal{D}$. We do this by cases. Recall that Thomass\'{e}'s conjecture has been proven for countable chains by Theorem \ref{Chaindich}. If $\mathcal{D}$ is finite or contains only finitely many non-trivial components each of which has one sibling, then $\mathcal{D}$ has one sibling by Lemma \ref{FinitenontrivialD}. If some component of $\mathcal{D}$ has continuum many siblings, the same holds for $\mathcal{D}$ itself by Lemma \ref{Infinitesiblingcomponent}. 

\begin{lemma} \label{Finitealeph0}
Let $\mathcal{D}=\bigoplus_{i < k} C_i \oplus \mathcal{N}$, $k > 0$, be countable where for every $ i < k$, $Sib(C_i)=\aleph_0$ and $\mathcal{N}$ is a DSC with only finitely many non-trivial components each of which has one sibling. Then $Sib(\mathcal{D})=\aleph_0$. 
\end{lemma} 

\begin{proof}
Let $\mathcal{D}'$ be a substructure of $\mathcal{D}$ such that $\mathcal{D}\approx \mathcal{D}'$. Then there is an embedding $\phi$ of $\mathcal{D}$ such that $\phi(\mathcal{D})\hookrightarrow \mathcal{D}' \hookrightarrow \mathcal{D}$. Since $\mathcal{D}$ has only finitely many non-trivial components, by Lemma \ref{InjectiononI} the embedding $\phi$ induces a finite permutation $\sigma$ on the non-trivial components. Moreover, all components in an orbit of $\sigma$ must be equimorphic.  Since every component of $\mathcal{N}$ has one sibling, and the sibling number of each $C_i$ is $\aleph_0$, we have $\sigma=\sigma_1\cdots \sigma_m\sigma_{m+1}\cdots \sigma_n$ where each $\sigma_i$, $1 \leq i \leq m$, is a permutation of the $C_i$ which are pairwise equimorphic and every $\sigma_i$, $m+1 \leq i \leq n$, is a permutation of the non-trivial components of $\mathcal{N}$ which are pairwise equimorphic. It follows that $\mathcal{D}'$ is of the form $\bigoplus_{j < k} C_j' \oplus \mathcal{N}'$ where each $C_j'$ is a sibling of some $C_i$ and $\mathcal{N} \cong \mathcal{N}'$ because $Sib(\mathcal{N})=1$ by Lemma \ref{FinitenontrivialD}. For every $ i < k$, let $\{C_{in}\}_{n < \omega}$ be a family of pairwise non-isomorphic siblings of $C_i$. For every function $f \in \mathbb{N}^k$ define
$$\mathcal{D}_f= \bigoplus_{i < k} C_{if(i)} \oplus \mathcal{N}.$$
By what we discussed, $\mathcal{D}'$ is isomorphic to some element of $\mathcal{S}=\{ \mathcal{D}_f : f \in \mathbb{N}^k \}$. For $f\neq g \in \mathbb{N}^k$ we might have $\mathcal{D}_f \cong \mathcal{D}_g$. It follows that $Sib(\mathcal{D})\leq |\mathcal{S}|=\aleph_0$. On the other hand, since $0 < k$, by Lemma \ref{Infinitesiblingcomponent} we have $Sib(\mathcal{D})\geq \aleph_0$.  Thus, $Sib(\mathcal{D})=\aleph_0$.  
\end{proof}

In case there are infinitely many finite non-trivial components, either there is a bound for their sizes or not.  Let $\mathcal{D}$ be a direct sum of chains. A sequence $(C_n)_{n < \omega}$ of components of $\mathcal{D}$ is called \textit{bounded} if there is $M < \omega$ such that $|C_n| < M$ for every $n < \omega$. Otherwise, it is \textit{unbounded}. We call a DSC {\em bounded} if there is a bound on all its components. For a positive integer $n$, we denote a chain of size $n$ by $C^n$.

\begin{lemma} \label{Countablebounded}
Let $\mathcal{D}$ be a countable DSC with infinitely many non-trivial components. If $\mathcal{D}$ is bounded, then $Sib(\mathcal{D})=\aleph_0$.
\end{lemma}

\begin{proof}
From the assumption we conclude that all components of $\mathcal{D}$ are finite. Let $\{C_i\}_{i\in I}$ be the set of non-trivial components of $\mathcal{D}$ and $M=\sup_i|C_i|$. Since $I$ is infinite, there exists some $2\leq n \leq M$ such that the number of components $C_i$ with $|C_i|=n$ is infinite. Without loss of generality let $n$ be maximum among such integers. Then for every $n < k \leq M$ the number of chains with size $k$ is finite. Set $\hat{\mathcal{D}}:=\bigoplus_\omega C^n\oplus \mathcal{F}$ where $\mathcal{F}$ is the finite direct sum of the components of $\mathcal{D}$ with size greater than $n$. Since the number of chains $C_i$ of size $n$ is infinite, $\mathcal{D}\approx \hat{\mathcal{D}}$.  Now, define
$$\Bar{\mathcal{D}}_{t_1, t_2, \ldots, t_{n-1}} :=  \bigoplus_{t_1} C^1 \oplus \bigoplus_{t_2} C^2 \oplus \cdots \oplus \bigoplus_{t_{n-1}} C^{n-1}\oplus \bigoplus_\omega C^n$$  
where for each $1\leq m \leq n-1$, $t_m \leq\omega$. For every $f \in \mathbb{N}^n$, define $\mathcal{D}_f := \Bar{\mathcal{D}}_{f(1), \ldots , f(n-1)}  \oplus \mathcal{F}$ which is a sibling of $\hat{\mathcal{D}}$ and consequently of $\mathcal{D}$. Let $f \neq g$ be such that $f(m)\neq g(m)$ for some $0 < m < n$. Then $\mathcal{D}_f$ and $\mathcal{D}_g$ have different numbers of components of size $m$ establishing that $\mathcal{D}_f \ncong \mathcal{D}_g$. Therefore, the collection $\mathcal{S}=\{ \mathcal{D}_f : f \in \mathbb{N}^n\}$ provides a countably infinite set of pairwise non-isomorphic siblings of $\mathcal{D}$ meaning that $Sib(\mathcal{D})\geq \aleph_0$. Now let $\mathcal{D}'$ be a sibling of $\hat{\mathcal{D}}$. Then $\mathcal{D}'\cong \Bar{\mathcal{D}}_{t_1,t_2,\ldots,t_{n-1}}\oplus\mathcal{F}$ for some $t_1, t_2, \ldots, t_{n-1}\leq\omega$. Since we have $\aleph_0$ such siblings, it follows that $Sib(\mathcal{D})=Sib(\hat{\mathcal{D}})\leq \aleph_0$. Hence, $Sib(\mathcal{D})=\aleph_0$. 
\end{proof}

\begin{lemma} \label{Strictlyinc}
Let $\mathcal{D}$ be a countable DSC. If some sibling of $\mathcal{D}$ contains a strictly increasing sequence of components, then  $Sib(\mathcal{D})=2^{\aleph_0}$. 
\end{lemma}

\begin{proof} 
Let $(C_n)_{n < \omega}$ be a strictly increasing sequence of components of $\mathcal{Q}$ where $\mathcal{Q} \approx \mathcal{D}$. It suffices to show that $Sib(\mathcal{Q})=2^{\aleph_0}$. 

Let $\mathcal{R}$ be the family of components $D$ of $\mathcal{Q}$ for which there exists some $m<\omega$ such that $D$ embeds in $C_m$. Let $\{D_1, D_2, \ldots\}$ be an enumeration of $\mathcal{R}$.  Denote by $H$ the direct sum of components of $\mathcal{Q}$ other than those ones contained in $\mathcal{R}$. Therefore, we can write $\mathcal{Q}=\bigoplus_{m<\omega}D_m \oplus H$. For every infinite $J\subseteq \omega$, define $\mathcal{Q}_J := \bigoplus_{n\in J}C_n \oplus H$. Since $\mathcal{R}$ contains the $C_n$, $\mathcal{Q}_J$ is a substructure of $\mathcal{Q}$, so we have $\mathcal{Q}_J \hookrightarrow \mathcal{Q}$. For every $m<\omega$, $D_m$ embeds in some $C_k$, and since $(C_n)_{n < \omega}$ is increasing, $D_m$ embeds in every $C_j$ where $k \leq j$. Since $J$ is infinite, there are infinitely many $n \in J$ such that $D_m \hookrightarrow C_n$. Let $f: \omega\to J$ be a function defined in this way: $f(0)=n_0$ where $n_0$ is the least element in $J$ for which $D_0\hookrightarrow C_{n_0}$; and for $m> 0$, $f(m)=n_m$ where $n_{m-1} < n_m$ is the least element of $J\setminus \{n_0, n_1, \ldots, n_{m-1}\}$ for which $D_m\hookrightarrow C_{n_m}$. Since the least element of each subset of $J$ is unique, the function $f$ is well-defined and since for $m\neq k$ we have $n_m\neq n_k$, $f$ is injective. The function $f$ induces an embedding from $\mathcal{Q}$ to $\mathcal{Q}_J$. Hence,  $\mathcal{Q} \approx \mathcal{Q}_J$. Now, for $J_1 \neq J_2 \subseteq \omega$ infinite, let $j$ be the least element in their symmetric difference. Without loss of generality assume that $j \in J_1$. Then $\mathcal{Q}_{J_1}$ has a component isomorphic to $C_j$ while $\mathcal{Q}_{J_2}$ does not because the sequence $(C_n)_{n < \omega}$ is strictly increasing, establishing that  $\mathcal{Q}_{J_1} \ncong \mathcal{Q}_{J_2}$. Since there are continuum many infinite subsets of $\omega$, we have $Sib(\mathcal{Q})\geq 2^{\aleph_0}$ and since $\mathcal{Q}$ is countable, we get  $Sib(\mathcal{Q})=2^{\aleph_0}$. 
\end{proof}

As a case in point, let $\mathcal{D}_{id}$ be the direct sum of exactly one chain of size $n$ for every $n < \omega$, that is $\mathcal{D}_{id}=\bigoplus_{n<\omega} C^n$. We see that although all components of $\mathcal{D}_{id}$ are finite, there is no bound for their sizes. 
Indeed, $\mathcal{D}_{id}$ itself has a strictly increasing sequence of finite chains. Therefore, $Sib(\mathcal{D}_{id})=2^{\aleph_0}$. Further, for a countable $\mathcal{D}$, if we can embed $\mathcal{D} \oplus \mathcal{D}_{id}$ in $\mathcal{D}$, then by applying Lemma \ref{Strictlyinc}, $Sib(\mathcal{D})=2^{\aleph_0}$.

\begin{lemma} \label{Increasingunbounded}
Let $\mathcal{D}$ be a countable DSC. Then, $\mathcal{D}$ contains an increasing and unbounded sequence of components if and only if $\mathcal{D} \oplus \mathcal{D}_{id} \approx \mathcal{D}$ where $\mathcal{D}_{id}=\bigoplus_{n < \omega} C^n$. 
\end{lemma}

\begin{proof}
First suppose that $\mathcal{D}$ contains an increasing and unbounded sequence $(C_n)_{n < \omega}$ of components. Then, $\mathcal{D}_{id} \hookrightarrow \bigoplus_{n < \omega} C_{2n+1}$. Let $\mathcal{F}$ be the direct sum of the components of $\mathcal{D}$ other than the $C_n$.  We have
$$\mathcal{D} \oplus \mathcal{D}_{id}= \bigoplus_{n < \omega} C_n \oplus \mathcal{D}_{id} \oplus \mathcal{F} \hookrightarrow \bigoplus_{n < \omega} C_{2n} \oplus \bigoplus_{n < \omega} C_{2n+1} \oplus \mathcal{F} = \mathcal{D}.$$
Therefore, $\mathcal{D} \oplus \mathcal{D}_{id} \approx \mathcal{D}$. 

For the converse, if $\mathcal{D} \oplus \mathcal{D}_{id} \approx \mathcal{D}$, then since $\mathcal{D}_{id}$ embeds in $\mathcal{D}$, $\mathcal{D}$ contains an increasing and unbounded sequence of components. 
\end{proof} 

In fact, when a direct sum of chains $\mathcal{D}$ contains an increasing and unbounded sequence of components, the sibling $\mathcal{D}\oplus \mathcal{D}_{id}$ of $\mathcal{D}$ contains a strictly increasing sequence of components.  
 
\begin{proposition} \label{Sibincreasingunbounded}
Let $\mathcal{D}$ be a countable DSC with an increasing and unbounded sequence of components, then $Sib(\mathcal{D})=2^{\aleph_0}$. 
\end{proposition} 

\begin{proof}
By Lemma \ref{Increasingunbounded}, $\mathcal{D} \approx \mathcal{D} \oplus \mathcal{D}_{id}$ and by Lemma \ref{Strictlyinc}, $Sib(\mathcal{D})=2^{\aleph_0}$. 
\end{proof}

For example, let $\mathcal{D}=\bigoplus_\omega \omega$. Then $\mathcal{D}$ contains an increasing and unbounded sequence of components. Hence, $Sib(\mathcal{D})=2^{\aleph_0}$. 
Now we are ready to state the main theorem of this section asserting that Thomass\'{e}'s conjecture holds for a countable DSC. 

\begin{theorem}  \label{ThomasseP}
If $\mathcal{D}$ is a countable DSC, then $Sib(\mathcal{D})=1$ or $\aleph_0$ or $2^{\aleph_0}$. 
\end{theorem}

\begin{proof}
Let $\mathcal{D}$ be a countable DSC. 

\noindent \textbf{Case 1:} $\mathcal{D}$ has finitely many non-trivial components. 

\textbf{Subcase 1.1 :} All components of $\mathcal{D}$ have one sibling. Then $Sib(\mathcal{D})=1$ by Lemma \ref{FinitenontrivialD}. 

\textbf{Subcase 1.2 :} $\mathcal{D}$ contains some components with sibling number $\aleph_0$ and all other components have just one sibling. Then $Sib(\mathcal{D})=\aleph_0$ by Lemma \ref{Finitealeph0}.  

\textbf{Subcase 1.3 :} $\mathcal{D}$ contains some component with the sibling number $2^{\aleph_0}$. Then $Sib(\mathcal{D})=2^{\aleph_0}$ by Lemma \ref{Infinitesiblingcomponent} and this fact that $Sib(\mathcal{D}) \leq 2^{\aleph_0}$. 

\noindent \textbf{Case 2 :} $\mathcal{D}$ has infinitely many non-trivial components $(C_n)_{n < \omega}$.

\textbf{Subcase 2.1 :} $\sup_n |C_n| < \omega$. Then $Sib(\mathcal{D})=\aleph_0$ by Lemma \ref{Countablebounded}. 

\textbf{Subcase 2.2 :} There exists an increasing and unbounded subsequence of $(C_n)_{n < \omega}$. Then, by Proposition \ref{Sibincreasingunbounded}, we have $Sib(\mathcal{D})= 2^{\aleph_0}$. 
\end{proof} 

We are also interested in classifying those countable direct sums of chains with sibling number $\aleph_0$. Conjecture \ref{Generalaleph0} gives the general form of such direct sums of chains; however, its proof relies on a positive answer to the following conjecture. 

\begin{conjecture} \label{Sumaleph0}
Let $\mathcal{D}= \mathcal{D}_1 \oplus \mathcal{D}_2$ where each $\mathcal{D}_i$ is a countable DSC with $Sib(\mathcal{D}_i)=\aleph_0$. Then $Sib(\mathcal{D}) \leq \aleph_0$.  
\end{conjecture} 
 

Assuming that the statement of the conjecture above holds, it also implies that when $\mathcal{D}=\bigoplus_{i < k} \mathcal{D}_i$ and $Sib(\mathcal{D}_i)=\aleph_0$ for every $i < k$, then $Sib(\mathcal{D})\leq \aleph_0$.

\begin{conjecture} \label{Generalaleph0}
Let $\mathcal{D}=\mathcal{D}_1 \oplus \mathcal{D}_2 \oplus \mathcal{D}_3$ be a countable DSC where 
\begin{enumerate}
    \item $\mathcal{D}_1$ is a bounded DSC,
    \item $\mathcal{D}_2$ is a non-empty and finite direct sum of chains  with sibling number $\aleph_0$, and
    \item $\mathcal{D}_3$ is a direct sum of chains with finitely many non-trivial components, each of which has one sibling. 
\end{enumerate}
   Then $Sib(\mathcal{D})=\aleph_0$. 
\end{conjecture}

We give a sketch of proof assuming that Conjecture \ref{Sumaleph0} is true.
We know that $Sib(\mathcal{D}_1)=1$ or $\aleph_0$ by Lemmas \ref{FinitenontrivialD} and \ref{Countablebounded}. Also, $Sib(\mathcal{D}_2)=\aleph_0$ by Lemma \ref{Finitealeph0} and $Sib(\mathcal{D}_3)=1$ by Lemma \ref{FinitenontrivialD}. If Conjecture \ref{Sumaleph0} holds, then $Sib(\mathcal{D})\leq \aleph_0$. Since $\mathcal{D}_2$ contains some component with sibling number $\aleph_0$, $Sib(\mathcal{D})\geq \aleph_0$ by Lemma \ref{Infinitesiblingcomponent}. Thus, $Sib(\mathcal{D})=\aleph_0$.

\section{Siblings of Arbitrary DSCs}

We know by Theorem \ref{Chaindich} that an arbitrary chain has one or infinitely many siblings. 
In this section, our aim is to prove that a direct sum of chains of any cardinality has one or infinitely many siblings. This proves the alternate Thomass\'{e} conjecture for an arbitrary direct sum of chains.

\begin{lemma}  \label{Infsibfinitetrivial}
Let $\mathcal{D}$ be a DSC with finitely many trivial components.  $\mathcal{D} \oplus C^1 \hookrightarrow \mathcal{D}$ if and only if $\mathcal{D}$ contains an increasing sequence of non-trivial components. In particular, if $\mathcal{D} \oplus C^1 \hookrightarrow \mathcal{D}$, then $Sib(\mathcal{D})=\infty$. 
\end{lemma} 

\begin{proof}
Suppose that $\mathcal{D} \oplus C^1$ embeds in $\mathcal{D}$ by some embedding $f$. Since the number of trivial components of $\mathcal{D}$ is finite and since $\mathcal{D}\oplus C^1$ has one more trivial component than $\mathcal{D}$, there is a trivial component, say $C_0$, of $\mathcal{D}$ such that $f(C_0) \subset C_1$ for some non-trivial component $C_1$ of $\mathcal{D}$. By Lemma \ref{InjectiononI}, $f(C_1) \subseteq C_2$ for some non-trivial component $C_2$ other than $C_1$. Iterating, we get an increasing sequence $(C_n)_{n < \omega}$ of non-trivial components of $\mathcal{D}$. 

For the converse, let $(C_n)_{n < \omega}$ be an increasing sequence of non-trivial components of $\mathcal{D}$. Then, embed the singleton $C^1$ in $C_0$ and for every $n < \omega$, embed $C_n$ in $C_{n+1}$. It follows that $\mathcal{D} \oplus C^1 \hookrightarrow \mathcal{D}$.  

The last conclusion comes from the lemma \ref{Increasingsequence}. 
\end{proof}

\begin{corollary} \label{Noembeddingtrivial}
Let $\mathcal{D}= \mathcal{N} \oplus T$ be a DSC where $\mathcal{N}$, resp $T$, is the direct sum of the non-trivial, resp trivial, components of $\mathcal{D}$ such that $T\neq\emptyset$. If $\mathcal{D}$ has no increasing sequence of non-trivial components, then $\mathcal{D} \not\hookrightarrow \mathcal{N}$. 
\end{corollary} 

\begin{proof}
If $\mathcal{N}=\emptyset$, then there is nothing to prove. Assume that $\mathcal{N} \neq \emptyset$. To seek a contradiction, suppose $\mathcal{N}\oplus T \hookrightarrow \mathcal{N}$. Let $x$ be an element of $T$. Then, $\mathcal{N} \oplus \{x\} \hookrightarrow \mathcal{N} \oplus T \hookrightarrow \mathcal{N}$, meaning that $\mathcal{N}$ contains an increasing sequence of non-trivial components. This contradiction completes the proof.  
\end{proof}

When $\mathcal{D}$ has countably many trivial components, Lemma \ref{Increasingsequence} asserts that the existence of only one increasing sequence of non-trivial components suffices to obtain infinitely many siblings for $\mathcal{D}$. The reason is that in this case there are at least countably many pairwise disjoint increasing sequences of non-trivial components which allows us to embed the trivial components into non-trivial ones. However, in the following lemma we say more.  

\begin{lemma}  \label{Pairwisedisincreasing}
Let $\mathcal{D}=\bigoplus_{i \in I}C_i \oplus T$ be a DSC where $T$ is the direct sum of the trivial components of $\mathcal{D}$ with  $|T|=\lambda>0$. There exist $\lambda$ pairwise disjoint increasing sequences of non-trivial components of $\mathcal{D}$ if and only if $\mathcal{D}\hookrightarrow \bigoplus_{i\in I} C_i$. In this case we have  $Sib(\mathcal{D})=\infty$.  
\end{lemma}

\begin{proof}
($\Rightarrow$) For every $\xi < \lambda$, let $(C_{\xi,n})_{n < \omega}$ be an increasing sequence of non-trivial components of $\mathcal{D}$ such that for  $\xi \neq \xi' < \lambda$, $(C_{\xi,n})_{n < \omega}$ and $(C_{\xi',n})_{n < \omega}$ are disjoint.  We can write $\mathcal{D}= \bigoplus_{\xi < \lambda} \bigoplus_{n < \omega} C_{\xi , n} \oplus \mathcal{N} \oplus T$ where $\mathcal{N}$ is the direct sum of non-trivial components other than the $C_{\xi,n}$. For each $\xi < \lambda$, we have 
$$\bigoplus_{n < \omega} C_{\xi,n} \oplus 1 \hookrightarrow  \bigoplus_{n < \omega} C_{\xi,n}.$$  
Since the sequences are pairwise disjoint, it follows that 
$$\bigoplus_{\xi < \lambda}\bigoplus_{n < \omega} C_{\xi,n} \oplus \mathcal{N} \oplus T = \bigoplus_{\xi < \lambda} \left ( \bigoplus_{n < \omega} C_{\xi,n} \oplus 1  \right ) \oplus \mathcal{N} \hookrightarrow  \bigoplus_{\xi < \lambda} \bigoplus_{n < \omega} C_{\xi,n} \oplus \mathcal{N}  = \bigoplus_{i \in I} C_i$$ 
as desired. 

($\Leftarrow$) Suppose $\mathcal{D} \hookrightarrow \bigoplus_{i\in I} C_i$. Let $f:\bigoplus_{i\in I} C_i\oplus T \to \bigoplus_{i\in I} C_i$ be an embedding and $T=\{ c_\xi : \xi < \lambda\}$. Let $\xi < \lambda$ be given. Let $C_{\xi,0}$ be the component of $\bigoplus_{i\in I} C_i$ such that $f(c_\xi)\in C_{\xi,0}$. For every $n < \omega$, let $C_{\xi,n+1}$ be the component of $\bigoplus_{i\in I} C_i$ such that $f(C_{\xi,n})\subseteq C_{\xi,n+1}$. Then $(C_{\xi,n})_{n < \omega}$ is an increasing sequence of components of $\bigoplus_{i\in I} C_i$. Let $\xi \neq \xi'$ and  $(C_{\xi,n})_{n < \omega}$ and $(C_{\xi',n})_{n < \omega}$ be the increasing sequences of components of $\bigoplus_{i\in I} C_i$ corresponding to $c_\xi$ and $c_{\xi'}$, respectively. If the sequences have a common component $C$, then two incomparable elements $c_\xi$ and $c_{\xi'}$ embed into $C$, a contradiction. Thus, we get $\lambda$ pairwise disjoint increasing sequences of non-trivial components of $\mathcal{D}$. 

In particular, when $\mathcal{D}\hookrightarrow \bigoplus_{i\in I} C_i$, then the $\bigoplus_{i\in I} C_i\oplus A^n$ form infinitely many pairwise non-isomorphic siblings of $\mathcal{D}$. 
\end{proof}

We give some examples of DSCs with infinitely many siblings.  For each $n > 1$, let $\omega_n$ be the least ordinal of cardinality $\aleph_n$. 

\begin{example}

\begin{enumerate}
    \item Consider Example \ref{2.1}, $\mathcal{D}=\bigoplus_{\omega_1} \omega \oplus \bigoplus_\omega (\omega + 1) \oplus T$, where $T$ is the direct sum of the singletons of $\mathcal{D}$ with $|T|=\aleph_1$.  We have $\aleph_1.\aleph_0=\aleph_1$. This means that $\mathcal{D}$ has $\aleph_1$ pairwise disjoint increasing sequences of components. Thus, in this case $\mathcal{D}$ satisfies the conditions of Lemma \ref{Pairwisedisincreasing}. Consequently, $Sib(\mathcal{D})=\infty$. 
    \item Consider Example \ref{2.2}, $\mathcal{D}=\bigoplus_{n < \omega} \bigoplus_{\omega_n} C^n \oplus T$, where $T$ is the direct sum of the singletons of $\mathcal{D}$ with $|T|=\aleph_\omega$. Then, for every $m < \omega$, let $S_m$ be a family of $\aleph_m$ pairwise disjoint sequences of chains $C^m$ which are increasing. Note that no sequence in $S_m$ has a common element with a sequence in $S_n$ when $m\neq n$. Therefore, the sequences in $\bigcup_{m<\omega} S_m$ are pairwise disjoint.  Moreover, we have  $\aleph_\omega=\bigcup_{m < \omega} \aleph_m$. Thus, there exist $\aleph_\omega$ pairwise disjoint increasing sequences of non-trivial components of $\mathcal{D}$. Again, Lemma \ref{Pairwisedisincreasing}  implies that $Sib(\mathcal{D})=\infty$. 
\end{enumerate} 
\end{example}

Let $\mathcal{D}$ be a DSC. 
For every $n < \omega$, let $\lambda_n$ be the number of components of $\mathcal{D}$ which are chains of size $n$. We denote the direct sum $\bigoplus_{\lambda_n} C^n$ by $\mathcal{C}^n_{\lambda_n}$.

\begin{proposition} \label{Generalpairwisedisincreasing}
Let $\mathcal{D}=\bigoplus_{k \leq m} \mathcal{C}^k_{\lambda_k} \oplus \mathcal{Q}$ be a DSC where $\mathcal{Q}$ is a direct sum of chains $C$ with $m < |C|$ satisfying one of the following:

\begin{enumerate}   
    \item $m > 1$ and there exist some $1 \leq i < j \leq m$ such that $\lambda_i \leq \lambda_j$ and $\lambda_j$ is infinite; 
    
    \item  $\lambda_1 > \lambda_2 > \cdots > \lambda_m \geq \aleph_0$, for every $m < n$, $\lambda_n$ is finite and $\mathcal{Q}$  contains an increasing sequence of components;

\end{enumerate}
then, $Sib(\mathcal{D})=\infty$.  
\end{proposition}

\begin{proof}
(1) Notice that $\mathcal{C}^i_{\lambda_i} \hookrightarrow \mathcal{C}^j_{\lambda_j}$ because $\mathcal{C}^j_{\lambda_j}$ contains $\lambda_i$ pairwise disjoint increasing sequences of components. It yields that 
$$\mathcal{D}_i=\bigoplus_{\substack {k=1 \\ k \neq i}}^m \mathcal{C}^k_{\lambda_k} \approx \bigoplus_{k=1}^m\mathcal{C}^k_{\lambda_k}.$$
This implies that $\mathcal{D} \approx \mathcal{D}_i \oplus \mathcal{Q}$. For every $n < \omega$, set $\mathcal{C}_n := \bigoplus_n C^i \oplus \mathcal{D}_i$. We have $\mathcal{C}_n \approx \mathcal{D}_i$ and consequently, $\mathcal{C}_n \oplus \mathcal{Q} \approx \mathcal{D}$. Since $\mathcal{C}_n$ contains exactly $n$ components of size $i$ and $\mathcal{Q}$ has no component of size $i$, the $\mathcal{C}_n\oplus \mathcal{Q}$ are pairwise non-isomorphic. Hence, $Sib(\mathcal{D})=\infty$. 

(2) We have $\mathcal{C}^{m+1}_{\lambda_{m+1}} \hookrightarrow \mathcal{Q}$ because $\lambda_{m+1}$ is finite and $\mathcal{Q}$ has an increasing sequence of components. Let $\mathcal{D}_{m+1}$ be the substructure of $\mathcal{D}$ with no component of size $m+1$. It follows that $\mathcal{D}_{m+1} \approx \mathcal{D}$. Set $\mathcal{C}_n := \bigoplus_n C^{m+1} \oplus \mathcal{D}_{m+1}$ for every $n < \omega$. We have $\mathcal{C}_n \approx \mathcal{D}_{m+1} \approx \mathcal{D}$. Since $\mathcal{C}_n$ contains exactly $n$ components of size $m+1$, the $\mathcal{C}_n$ are pairwise non-isomorphic. Therefore, $Sib(\mathcal{D})=\infty$. 
\end{proof}

\begin{example}
\begin{enumerate}
    \item Consider Example \ref{2.3}, $\mathcal{D}=\bigoplus_{\omega_1} C^1 \oplus \bigoplus_\omega C^3$. We have $\lambda_1 = \aleph_1$, $\lambda_2=0$ and $\lambda_3=\aleph_0$. Notice that $\lambda_2 \leq \lambda_3$ and $\lambda_3$ is infinite. Therefore, $Sib(\mathcal{D})=\infty$. 
    
    \item Consider Example \ref{2.4}, $\mathcal{D}=\bigoplus_{\omega_3} C^1 \oplus \bigoplus_{\omega_2} C^2 \oplus \bigoplus_{\omega_1} C^3 \oplus \bigoplus_{\omega_1} \omega$. We have $\lambda_1 = \aleph_3$, $\lambda_2 = \aleph_2$ and $\lambda_3 = \aleph_1$. Set $\mathcal{Q}=\bigoplus_{\omega_1} \omega$. We have $3 < |\omega|$. Since $\mathcal{Q}$ contains $\lambda_3$ pairwise disjoint increasing sequences of components, it follows that $Sib(\mathcal{D})=\infty$. 
\end{enumerate}
\end{example}

In the following lemma when $\mathcal{D}$ has no increasing sequence of non-trivial components, then it has one sibling, on condition that all its components have just one sibling.  For simplicity of its proof, we need the following setting. 
 
Let $\mathcal{D}=\bigoplus_{i \in I}C_i$ be a DSC. For $i, j \in I$ define $i \sim j$ if $C_i \approx C_j$. Then $\sim$ is an equivalence relation on $I$.  So, we can  write $I$ as a sequence of its equivalence classes that is $I=\bigcup_{\xi < \gamma} I_\xi$ where $\gamma$ is the number of equivalence classes of $I$ and for every $\xi < \gamma$, $I_\xi = [i_\xi]_\sim$ for some $i_\xi \in I$. For two disjoint equivalence classes $[i]_\sim$ and $[j]_\sim$ , $i, j \in I$, define $[i]_\sim < [j]_\sim$ if $C_i \hookrightarrow C_j$ and $C_j \not\hookrightarrow C_i$. It is clear that if $[i]_\sim < [j]_\sim$, then every chain $C_{i'}$ where $i' \in [i]_\sim$, embeds in any chain $C_{j'}$, $j' \in [j]_\sim$. Therefore, when there exists a strictly increasing sequence of components of $\mathcal{D}$ w.r.t embeddability, then there is an increasing sequence of equivalence classes of $I$ w.r.t $<$ and vice versa.

\begin{lemma} \label{Noincreasing}
Let $\mathcal{D}$ be a DSC such that there is no increasing sequence of its non-trivial components w.r.t embeddability and  every component has just one sibling. 
Then $Sib(\mathcal{D})=1$. 
\end{lemma}

\begin{proof} 
Let $\mathcal{D}=\bigoplus_{i \in I} C_i \oplus T$, where $T$ is the direct sum of the trivial components of $\mathcal{D}$, satisfy the conditions of the lemma. Let $f$ be an embedding of $\mathcal{D}$ and $\hat{f} : I \to I$ be defined as in Lemma \ref{InjectiononI}. Let $i \in I$ be given. Since there is no increasing sequence of non-trivial components of $\mathcal{D}$, there exist some $m , n$ such that $m < n$ and $\hat{f}^m(i)=\hat{f}^n(i)$. Assume that $n$ is the least integer with above property. We prove that $m=0$. Suppose $0 < m$. By setting $\hat{f}^m(i)=\hat{f}^n(i)=j$, we have $f(f^{m-1}(C_i))\subseteq C_j$ and $f(f^{n-1}(C_i))\subseteq C_j$. Notice that $0 \leq m-1 < n-1$. So $f^{m-1}(C_i)$ and $f^{n-1}(C_i)$ are two distinct components of $\mathcal{D}$. But $f$ embeds these two components in $C_j$ which is not possible because $\hat{f}$ is injective. For $i, j \in I$ define $i \equiv_f j$ if there exists some $m$ such that $\hat{f}^m(i)=j$. The relation $\equiv_f$ is an equivalence relation and for every $i \in I$ there is some $m(i)$ such that $[i]_\sim = [i_0]_{\equiv_f} \cup \cdots \cup [i_{m(i)}]_{\equiv_f}$ where each $[i_k]_{\equiv_f}$, $0 \leq k \leq m(i)$, is finite. Further, $i \equiv_f j$ implies that $[i]_\sim = [j]_\sim$ and since every component has only one sibling, we have $C_i \cong C_j$.   

Let $\mathcal{D}'\subseteq \mathcal{D}$ be a sibling of $\mathcal{D}$. Then, there is an embedding $f$ of $\mathcal{D}$ such that $f(\mathcal{D})  \hookrightarrow \mathcal{D}' \hookrightarrow \mathcal{D}$. Let $i \in I$ be given. We have $C_i \cap \mathcal{D}' \cong C_j \cap \mathcal{D}' \cong C_i$ where $j \in [i]_{\equiv_f}$. Let $|[i]_{\equiv_f}|=m_i$. Then, $\mathcal{D}'$ contains $m_i$ copies of $C_i$ as does $\mathcal{D}$ and they are pairwise isomorphic.  Further, if for some $x \in T$, $f(x)$ is contained in a non-trivial component $C$, then either there is an increasing sequence of non-trivial components, a contradiction, or $C$ embeds in $x$ which is impossible. So, we have $f(T) \subseteq T$ and $|f(T)|=|T|$. Therefore, $\mathcal{D}'$ is of the form $\bigoplus_{j \in J} C_j \oplus T'$ where $f(T) \subseteq T' \subseteq T$. It follows that $\bigoplus_{j \in J} C_j \cong \bigoplus_{i \in I} C_i$ and $|T'|=|T|$ meaning that $\mathcal{D}' \cong \mathcal{D}$. Thus,  $Sib(\mathcal{D})=1$.   
\end{proof}

In Lemma \ref{Noincreasing}, the lack of increasing sequence of non-trivial components prevents embedding trivial components into non-trivial ones by Corollary \ref{Noembeddingtrivial} because it  does not allow embedding a singleton into a non-trivial component. This means that any sibling $\mathcal{D}'$ of $\mathcal{D}$ with the conditions in the Lemma \ref{Noincreasing} has a set of trivial components $T'$ with $|T'|=|T|$ where $T$ is the direct sum of the trivial components of $\mathcal{D}$. The same happens if the cardinal of $T$ is uncountable but $\mathcal{D}$ has countably many pairwise disjoint increasing sequences of non-trivial components. However, in the latter case, if the number of components of $\mathcal{D}$ which are chains of size 2 is countable, we can apply Proposition \ref{Generalpairwisedisincreasing}, part (2), to obtain infinitely many pairwise non-isomorphic siblings. Remember that for every $n < \omega$, $\lambda_n$ is the number of components of $\mathcal{D}$ which are chains of size $n$. Bearing this in mind, we may assume that there exists some $m < \omega$ for which $\lambda_1 > \lambda_2 > \cdots > \lambda_m \geq \aleph_0$ and for every $m < n$, $\lambda_n$ is finite. Moreover, assume that there is no increasing sequence of components $\mathcal{D}$ with size greater than $m$. We have the following lemma which is a generalisation of Lemma \ref{Noincreasing} as well as Lemma  \ref{FinitenontrivialD}.     

\begin{lemma} \label{Generalnoincreasing}
Let $\mathcal{D}=\bigoplus_{k \leq m} \mathcal{C}^k_{\lambda_k} \oplus \mathcal{Q}$ where $\lambda_1 > \lambda_2 > \cdots > \lambda_m \geq \aleph_0$, for every $m < n$, $\lambda_n$ is finite and $\mathcal{Q}$ is a direct sum of chains $C$ with $m < |C|$ such that it does not contain increasing sequence of components. If each component of $\mathcal{D}$ has one sibling, 
then $Sib(\mathcal{D})=1$. 
\end{lemma}

\begin{proof}
Let $\mathcal{Q}=\bigoplus_{i \in I} C_i$ where for every $i \in I$, $m < |C_i|$ and $Sib(C_i)=1$. By Lemma \ref{Noincreasing},  $Sib(\mathcal{Q})=1$. Let $\mathcal{D}'\subseteq \mathcal{D}$ be sibling of $\mathcal{D}$. Then, there is an embedding $f$ of $\mathcal{D}$ such that $f(\mathcal{D}) \hookrightarrow \mathcal{D}' \hookrightarrow \mathcal{D}$. Now for no $n < \omega$ and $1 \leq k \leq m$, $f^n(C^k)$ is a subchain of a component $C$ of $\mathcal{Q}$, because otherwise either there is an increasing sequence of components of $\mathcal{Q}$ or $C$ embeds in $C^k$ which is impossible because $|C| > |C^k|$. Moreover, for no distinct $1 \leq i , j \leq m$, $\mathcal{C}^i_{\lambda_i} \hookrightarrow \mathcal{C}^j_{\lambda_j}$ because if $i < j$, then $\lambda_i > \lambda_j$ and if $i > j$, then chains of size $i$ cannot embed into chains of size $j$. It follows that $\mathcal{D}'$ is of the form $\mathcal{D}'=\mathcal{Q}' \oplus \bigoplus_{k \leq m} \mathcal{C}^k_{\lambda_k}$ where $\mathcal{Q}' \approx \mathcal{Q}$. We know that $\mathcal{Q}\cong \mathcal{Q}'$. Therefore, $\mathcal{D} \cong \mathcal{D}'$ establishing that $Sib(\mathcal{D})=1$.
\end{proof}

A theorem due to Sierpi\'{n}ski \cite{ROS} (see also \cite{DM}) states that there are $2^{2^{\aleph_0}}$ dense subchains $C$ of $\mathbb{R}$ of cardinality $2^{\aleph_0}$ which are pairwise non-equimorphic and they are also \textit{embedding rigid} meaning that the identity map is the only embedding of $C$ into itself. This means that we can easily construct a direct sum of chains satisfying Lemma \ref{Generalnoincreasing} with infinitely many pairwise non-equimorphic components each of which has only one sibling. 

Now we are ready to conclude the alternate Thomass\'{e} conjecture for a direct sum of chains. 

\begin{theorem} \label{AltThomasseP}
If $\mathcal{D}$ is a direct sum of chains, then $Sib(\mathcal{D})=1$ or  $\infty$. 
\end{theorem}

\begin{proof}
\textbf{Case 1} $\mathcal{D}$ contains a component with infinitely many siblings. Then $Sib(\mathcal{D})=\infty$ by Lemma \ref{Infinitesiblingcomponent}.  

\noindent \textbf{Case 2:} Every component of $\mathcal{D}$ has just one sibling. 

\textbf{Subcase 2.1} $\mathcal{D}$ has finitely many non-trivial components. Then $Sib(\mathcal{D})=1$ by Lemma \ref{FinitenontrivialD}.

\textbf{Subcase 2.2}   $\mathcal{D}$ has infinitely many non-trivial components. Let $\mathcal{D}=\bigoplus_{k \leq m}\mathcal{C}^n_{\lambda_n} \oplus \mathcal{Q}$ where  $\mathcal{Q}$ is a direct sum of chains $C$ with $m < |C|$.

\textbf{Subcase 2.2.1}   If $m >1$ and for some $1 \leq i < j \leq m$, $\lambda_i \leq \lambda_j$ such that $\lambda_j$ is infinite; or if $\lambda_1 > \lambda_2 > \cdots > \lambda_m \geq \aleph_0$, for every $m < n$, $\lambda_n$ is finite and $\mathcal{Q}$ contains an increasing sequence of components, then $Sib(\mathcal{D})=\infty$ by Proposition \ref{Generalpairwisedisincreasing}. 

\textbf{Subcase 2.2.2} If  $\lambda_1 > \lambda_2 > \cdots > \lambda_m \geq \aleph_0$, for every $m < n$, $\lambda_n$ is finite and $\mathcal{Q}$ does not have increasing sequence of components, then $Sib(\mathcal{D})=1$ by Lemma   \ref{Generalnoincreasing}.    
\end{proof}










\section{Extensions}

Remember that in order to prove the alternate Thomass\'e conjecture for countable direct sums of chains (Theorem \ref{DichotomyPcountable}), we just used the wqo property of countable chains and some structural results in Section \ref{Structural}. 
We conjecture that Theorem \ref{DichotomyPcountable} might be generalised to direct sums of those posets which belong to a wqo. For instance, trees can be studied in order theoretical sense. Let $P$ be a poset and $x\in P$. We denote by $P_x$ the set $\{z\in P : z\leq_P x\}$. In this context, a {\em tree} is a poset $T$ such that for every $x\in T$, the set $T_x$ is a chain and any two elements of $T$ are compatible. The greatest lower bound of two elements $x, y$ of a tree $T$, denoted by $x\wedge y$, is called the {\em meet} of $x$ and $y$. A {\em meet-tree} is a tree in which any two elements have a meet. A $\wedge$-{\em embedding} from a meet-tree $T$ to another meet-tree $T'$ is an injective order-preserving map $f:T\to T'$ such that $f(x\wedge y)=f(x)\wedge f(y)$ for all $x, y\in T$. 
A theorem due to Kruskal \cite{KR} asserts that the class of finite trees is wqo with respect to $\wedge$-embedding. Moreover, each finite tree has only one sibling. With these facts, it is easy to show that a direct sum of finite trees has one or infinitely many siblings. However, verifying Thomass\'e's conjecture for a direct sum of finite trees might need some efforts similar to what we did in Section \ref{CountableDSC}. 
Also, there is a result due to Corominas \cite{C} asserting that the class of countable trees is  wqo with respect to $\wedge$-embedding. Therefore, $\wedge$-embedding is a well-founded relation on this class meaning that we can use induction. Clearly, a singleton has one sibling. Hence, we may ask the following. 

\begin{problem}
Let $\mathcal{T}=\bigoplus_{i\in I} T_i$ be a direct sum of countable trees where I is countable such that $Sib(T_i)=1$ or $\infty$ for each $i\in I$. 
\begin{enumerate}
    \item Is it true that $Sib(\mathcal{T})=1$ or $\infty$? 
    \item What can we say about $Sib(\mathcal{T})$?
\end{enumerate}

\end{problem}

Direct sum of posets is a specific case of a more general notion of {\em direct sum of binary relations}. Let $\{\mathcal{E}_i\}_{i\in I}$ be a family of binary relations over a fixed binary relational symbol $R$, that is, $\mathcal{E}_i=(M_i,R)$ for each $i\in I$. Their {\em direct sum}, denoted by $\mathcal{E}:=\bigoplus_{i\in I}\mathcal{E}_i$, is the disjoint union of the $M_i$ such that we have $(a,b)\in R^\mathcal{E}$ if and only if there exists $i\in I$ such that $a, b\in M_i$ and $(a, b)\in R^{\mathcal{E}_i}$. Then, we call each $\mathcal{E}_i$ a {\em component} of $\mathcal{E}$. It is of interest to ask the same questions as in the problem above for a direct sum of binary relations whose components belong to a wqo as well as for a direct sum of binary relations for which the alternate Thomass\'e conjecture or Thomass\'e's conjecture itself has been verified. For instance, direct sum of rayless trees, direct sum of rayless graphs, direct sum of rooted trees, direct sum of scattered trees, etc. 

\vspace{0.5cm}
\noindent{\bf \large Acknowledgements}

I would like to thank my Ph.D. supervisors Professor Robert Woodrow and Professor Claude Laflamme for suggesting this problem and for their help and advice. 

%
%

\vspace{1cm}

\textsc{Department of Mathematics and Statistics, University of Calgary, Calgary, Alberta, Canada, T2N 1N4}

{\em Email address:} \texttt{davoud.abdikalow@alumni.ucalgary.ca}

\end{document}